\documentclass[12pt, reqno]{amsart}
\usepackage{graphicx}
\usepackage{subfigure}
\usepackage{latexsym}
\usepackage{amsmath}
\usepackage{amssymb}
\usepackage{color}

 \newtheorem{theorem}{Theorem}[section]
 
 \newtheorem{Prop}[theorem]{Proposition}
 \newtheorem{Lem}[theorem]{Lemma}

\newcommand{\ba}{\begin{array}}
\newcommand{\ea}{\end{array}}
\newcommand{\beq}{\begin{equation}}
\newcommand{\eeq}{\end{equation}}

 \numberwithin{equation}{section}
 
\begin{document}
\title{Topological properties of self-similar fractals with one parameter}

\author{Jun Jason Luo} 
\address{College of Mathematics and Statistics, Chongqing University,  401331 Chongqing, China
	\newline\indent Institut f\"ur Mathematik, Friedrich-Schiller-Universit\"at Jena, 07743 Jena, Germany}
\email{jasonluojun@gmail.com}

\author{Lian Wang} \address{College of Mathematics and Statistics,  Chongqing Uinversity, Chongqing, 401331, P.R. China} 
\email{lwang@cqu.edu.cn}


\thanks{The research is supported by the NNSF of China (No.11301322), the Fundamental and Frontier Research Project of Chongqing (No.cstc2015jcyjA00035)}

\date{\today}

\begin{abstract}
 In this paper, we study two classes of planar self-similar fractals $T_\varepsilon$ with a shifting  parameter $\varepsilon$. The first one is a class of self-similar tiles by shifting $x$-coordinates of some digits. We give a detailed discussion on the disk-likeness ({\it i.e., the property of being a topological disk}) in terms of $\varepsilon$. We also prove that $T_\varepsilon$ determines a quasi-periodic tiling if and only if $\varepsilon$ is rational. The second one is a class of self-similar sets by shifting diagonal digits. We give a necessary and sufficient condition for $T_\varepsilon$ to be connected.
\end{abstract}

\maketitle

\section{\bf Introduction}
Let $A$ be a $d\times d$ integer expanding matrix (i.e., all of its eigenvalues are strictly larger than one in modulus), let ${\mathcal D}=\{d_1,\dots,d_N\}\subset {\mathbb R}^d$ be a digit set with $N=|\det(A)|$. Then we can define an {\it iterated function system (IFS)} $\{S_{j}\}_{j=1}^{N}$ where $S_j$ are affine maps
\[
S_{j}(x)=A^{-1}(x+d_{j}),  \quad x\in\mathbb{R}^d.
\]
Since $A$ is expanding, each $S_{j}$ is a contractive map under a suitable norm \cite{LaWa} of $\mathbb{R}^{d}$, there is a unique nonempty compact subset $T:=T(A,{\mathcal D})\subset{\mathbb{R}}^d$  \cite{Hu}  such that
$$T=\bigcup_{j=1}^NS_j(T)=A^{-1}(T+{\mathcal D}).$$
The set $T$ also has the radix expansion
\begin{equation}\label{eq1.1}
T=\left\{\sum^{\infty}_{k=1}A^{-k}d_{j_k}:d_{j_k}\in\mathcal{D}\right\}.
\end{equation}
We call $T$ a {\it self-affine set} generated by the pair $(A, {\mathcal D})$ (or the IFS $\{S_{j}\}_{j=1}^{N}$). Moreover, if $T$ has non-void interior (i.e., $T^\circ\ne\emptyset$), then there exits a discrete set ${\mathcal J}\subset {\mathbb R}^d$ satisfying $$T+{\mathcal J}={\mathbb R}^d \quad \text{and}\quad (T^\circ+t)\cap (T^\circ +t')=\emptyset \quad \text{with}\quad t\ne t', t,t'\in {\mathcal J}.$$
We call such $T$ a {\it self-affine tile} and $T+{\mathcal J}$ a {\it tiling} of ${\mathbb R}^d$. In particular, if $A$ is a similarity, then $T$ is called a {\it self-similar set/tile}.

Since the fundamental theory of self-affine tiles was established by Lagarias and Wang (\cite{LaWa},\cite{LaWa1},\cite{LaWa2}), there have been considerable interests in the topological structure of self-affine tiles $T$,  including but not limited to the connectedness of  $T$ (\cite{GrHa},\cite{HaSaVe},\cite{KiLa},\cite{AkGj},\cite{DeLa}),  the boundary $\partial T$ (\cite{AL},\cite{LeLu3},\cite{LAT}), or the interior $T^\circ$ of a connected tile $T$  (\cite{NT},\cite{NT2}). Especially in $\mathbb{R}^2$, the study on the disk-likeness of $T$   ({\it i.e., the property of being a topological disk}) has attracted a lot of attentions  (\cite{BaWa},\cite{LeLa},\cite{LRT},\cite{Ki},\cite{DeLa}). For other related works, we refer to \cite{LLY},\cite{LeLu},\cite{LeLu2},\cite{LLX}
 and a survey paper \cite{AT}.

Any change on the matrix $A$ and the digit set $\mathcal{D}$ may lead to some change on the topology of $T(A,\mathcal{D})$. To simplify the analysis on the relations between those two types of ``changes'', one may fix an expanding matrix $A$ and focus on particular choices of the digit set $\mathcal{D}$. Recently Deng and Lau \cite{DeLa} considered a class of planar self-affine tiles $T$ that are generated by a lower triangular expanding matrix and  product-form digit sets. They gave a complete characterization on both connectedness and disk-likeness of $T$.

Motivated by the above results, in this paper, we investigate the topological properties of the following two classes of self-similar fractals in $\mathbb{R}^2$. Assume that $A$ is a diagonal matrix with equal nonzero entries, hence $A$ is a similarity. In the first class, we consider a kind of  digit sets ${\mathcal D}_\varepsilon$ with a shift $\varepsilon$  on the $x$-coordinates of some digits. We obtain an analogous result to \cite{DeLa}.

\begin{theorem}\label{mainthm}
Let $p$ be an integer with $|p|=2m+1$ where $m\in {\mathbb N}$, let $\varepsilon\in {\mathbb R}$. Suppose $T_\varepsilon$ is the self-similar set generated by  $A=\left[\begin{array}{cc}
p& 0 \\
0 & p
\end{array}\right]$ and
$${\mathcal D}_\varepsilon=\left\{\left[\begin{array}{c}
i+b_j\\
j
\end{array}\right]: b_j=\frac{1-(-1)^j}{2}\varepsilon, \quad i, j\in \{0,\pm 1, \dots, \pm m\}\right\}.$$ Then $T_\varepsilon$ is a self-similar tile. Moreover,

(i) if $|\varepsilon|<|p|$, then $T_\varepsilon$ is disk-like;

(ii) if $|p|^n\leq |\varepsilon|<|p|^{n+1}$ for $n\geq 1$, then $T_{\varepsilon}^\circ$ has $|p|^n$ components and every closure of the component is disk-like. (see Figure \ref{fig1})
\end{theorem}

In fact, Theorem \ref{mainthm} can be proved in a more general setting where $A=\left[\begin{array}{cc}
p& 0 \\
0 & q
\end{array}\right]$ without causing much difficulty. We omit this for the completeness of the paper.

Moreover, we further consider the quasi-periodic tiling property of  $T_\varepsilon$ (the definition will be recalled in Section 3). Let ${\mathcal D}_{\varepsilon,k}={\mathcal D}_\varepsilon+ A{\mathcal D}_\varepsilon+\cdots+ A^{k-1}{\mathcal D}_\varepsilon$ and ${\mathcal D}_{\varepsilon,\infty}=\bigcup_{k=1}^\infty{\mathcal D}_{\varepsilon,k}$. We prove that

\begin{theorem}\label{mainthm2}
With the same $(A, {\mathcal D}_\varepsilon)$ as in Theorem \ref{mainthm}, $T_\varepsilon+{\mathcal D}_{\varepsilon,\infty}$ is a quasi-periodic tiling if and only if $\varepsilon$ is a rational number.
\end{theorem}

The second one is a class of self-similar sets  $T_\varepsilon$ with a shift $\varepsilon$ on  the diagonal digits along the diagonal line. In contrast with the first one, in this class, $T_\varepsilon$ might not be a tile, as the open set condition will not always hold for any $\varepsilon$ (see the remark at the end of Section 4). Let $\delta_{ij}=1$ if $i=j$; $\delta_{ij}=0$ if $i \ne j$. Then we have

\begin{theorem}\label{mainthm3}
Let $p$ be an integer with $|p|>2$, $\varepsilon\in {\mathbb R}$. Suppose $T_\varepsilon$ is the self-similar set generated by  $A=\left[\begin{array}{cc}
p& 0 \\
0 & p
\end{array}\right]$ and
$${\mathcal D}_\varepsilon=\left\{\left[\begin{array}{c}
i+a_{ij}\\
j+a_{ij}
\end{array}\right]: a_{ij}=\delta_{ij}\varepsilon, \quad
   i, j\in \{0, 1, \dots,|p|-1\}\right\}.$$
Then $T_\varepsilon$ is connected if and only if $|\varepsilon|\leq \frac{(|p|-1)^{2}}{|p|-2}$. (see Figure \ref{fig.3})
\end{theorem}

For the organization of the paper, we prove Theorem \ref{mainthm} in Section 2, Theorem \ref{mainthm2} in Section 3 and Theorem \ref{mainthm3} in Section 4 respectively.

\section{\bf Self-similar tiles}
Let $(A, {\mathcal D}_\varepsilon)$  be the pair as in Theorem \ref{mainthm}, and let $T_\varepsilon:=T(A, {\mathcal D}_\varepsilon)$ be the associated self-similar set. It suffices to prove the theorem for $p>0$, otherwise we can replace $A$ by $A^2$ according to the fact
$$T_\varepsilon=A^{-1}(T_\varepsilon+{\mathcal D}_\varepsilon)=A^{-2}(T_\varepsilon+{\mathcal D}_\varepsilon+A{\mathcal D}_\varepsilon)$$ where  the digit set ${\mathcal D}_\varepsilon+A{\mathcal D}_\varepsilon$ can be written as:
\begin{eqnarray*}
{\mathcal D}_\varepsilon+A{\mathcal D}_\varepsilon &=& \left\{\left[\begin{array}{c}
pr+l+(pb_k +b_t)\\pk+t\end{array}\right]:r,l,k,t\in\{0,\pm1,\dots,\pm m\}\right\} \\
&=& \left\{\left[\begin{array}{c}i+b'_j\\j\end{array}\right]:i,j\in\{0,\pm1,\dots,\pm(2m^2+2m)\}\right\},
\end{eqnarray*}
where $b'_j=pb_k +b_t\ \text{with}\ j=pk+t$.

We denote by ${\mathcal I}$ the set of ${\mathbf i}=i_1i_2\cdots$ with $i_n\in \{0,\pm1,\dots,\pm m\}$. In view of (\ref{eq1.1}),
\begin{equation}\label{eq2.1}
T_\varepsilon=\left\{ \left[\begin{array}{c}
p({\mathbf i})+b({\mathbf j})\\
p({\mathbf j})
\end{array}\right]: {\mathbf i}=i_1i_2\cdots,{\mathbf j}=j_1j_2\cdots\in {\mathcal I}  \right\}
\end{equation}
where $$p({\mathbf i})=\sum_n \frac{i_n}{p^n},\quad b({\mathbf j})=\sum_n \frac{b_{j_n}}{p^n} \quad \text{and} \quad p({\mathbf j})=\sum_n \frac{j_n}{p^n}.$$

 It follows from the above that the range of the $y$-coordinate of $T_\varepsilon$ is the interval $[-\frac{1}{2},\frac{1}{2}]$. For each fixed $y=p({\mathbf j})$ such that the radix expansion is unique, then the horizontal cross section of $T_\varepsilon$ is an interval of length $1$ with endpoints at $-\frac{1}{2}+b({\mathbf j})$ and $\frac{1}{2}+b({\mathbf j})$; for the other $y$-coordinate that has two radix expansions, the horizontal cross section of $T_\varepsilon$ is the union of two intervals with length $1$.

The following lemma is essentially the same as Proposition 2.2 in \cite{DeLa}.

\begin{Lem}\label{lem2.1}
$T_\varepsilon$ is a self-similar tile. Moreover, for any sequence $\{\ell_s\}_{s\in {\mathbb Z}}$ in ${\mathbb R}$, let ${\mathcal J}=\{(n+\ell_s,s)^t: n,s\in{\mathbb Z}\}$. Then $T_\varepsilon+{\mathcal J}$ is a tiling of ${\mathbb R}^2$.
\end{Lem}

\begin{proof}
Let $D=\{0,\pm1,\dots, \pm m\}$. For any $(x,y)^t\in {\mathbb R}^2$, since $T(p,D)=[-\frac{1}{2},\frac{1}{2}]$, we can find $s\in {\mathbb Z}$ such that $y-s\in [-\frac{1}{2},\frac{1}{2}]$. Let ${\mathbf j}\in {\mathcal I}$ such that $y-s=p({\mathbf j})$.  On the other hand, there is $n\in {\mathbb Z}$ such that $x-b({\mathbf j})-\ell_s-n\in [-\frac{1}{2},\frac{1}{2}]$. This implies that $x-b({\mathbf j})-\ell_s-n=p({\mathbf i})$ for some ${\mathbf i}\in {\mathcal I}$. It follows that $(x,y)^t\in T_\varepsilon+(n+\ell_s,s)^t$. Hence $T_\varepsilon+{\mathcal J}={\mathbb R}^2$.

Note that for almost all $y\in {\mathbb R}$, the above $s, {\mathbf j}$ are unique. If we fix such $y$, then for almost all $x\in {\mathbb R}$, the above $n$ is also unique. Therefore, for almost all $(x,y)^t\in {\mathbb R}^2$, the above $n,s$ are unique. Hence $\{T_\varepsilon+t: t\in{\mathcal J}\}$ are measure disjoint sets. That means
$T_\varepsilon+{\mathcal J}$ tile ${\mathbb R}^2$.
\end{proof}

Geometrically, the tile $T_\varepsilon$ has two sides on the horizontal line $y=-\frac{1}{2}$ and $y=\frac{1}{2}$ with length one. Lemma \ref{lem2.1} implies that the tiling can be moved horizontally. The following is an elementary criterion for connectedness.

\begin{Lem}[\cite{Ha},\cite{KiLa}]\label{lem2.2}
Let $\{S_j\}^N_{j=1}$ be an IFS of contractions on $\mathbb{R}^d$ and let $K$ be its attractor. Then $K$ is connected if and only if, for any $i\neq j\in\{1,2,\dots,N\}$, there exits a sequence $i=j_1,j_2,\dots,j_n=j$ of indices in $\{1,2,\dots,N\}$ so that $S_{j_k}(K)\cap S_{j_{k+1}}(K)\neq\emptyset$ for all $1\leq k<n$.
\end{Lem}

Let $$f_{i,j}\left(\left[\begin{array}{c}
x\\
y
\end{array}\right]\right)=A^{-1}\left(\left[\begin{array}{c}
x\\
y
\end{array}\right]+\left[\begin{array}{c}
i+b_j\\
j
\end{array}\right]\right),$$ where $ i,j\in\{0,\pm1,\cdots,\pm m\}.$  Then $f_{i,j}$'s form an IFS which generates $T_\varepsilon$. By (\ref{eq2.1}), the elements of $f_{i,j}(T_\varepsilon)$ are of the form
\begin{equation}\label{eq2.2}
\left[\begin{array}{c}
p(i{\mathbf i})+b(j{\mathbf j})\\
p(j{\mathbf j})
\end{array}\right]
\end{equation} where ${\mathbf i}=i_1i_2\cdots,\ {\mathbf j}=j_1j_2\cdots\in {\mathcal I}$. For $j_1j_2\cdots j_n$ with $j_t\in \{0,\pm1,\dots, \pm m\}$,  if we denote
\begin{equation}\label{eq2.3}
G_{j_1\cdots j_n}=\bigcup_{i_1}\cdots\bigcup_{i_n}f_{i_1,j_1}\circ f_{i_2, j_2}\circ\cdots \circ f_{i_n, j_n}(T_\varepsilon),
\end{equation}
then $$T_\varepsilon=\bigcup_{j_1,\dots, j_n}G_{j_1j_2\cdots j_n}.$$

For simplicity of our statements, we write ${\mathbf i}_0=i_1\cdots i_n,\ {\mathbf j}_0=j_1\cdots j_n$ and ${\mathbf 0}=\underbrace{0\cdots 0}_n$ where $n\geq 1$.

\begin{Prop}\label{prop1}
For ${\mathbf j}_0=j_1\cdots j_n\in \{0,\pm1,\dots,\pm m\}^n$ and $k,\ell\in\{0,\pm1,\dots,\pm m\}$,

(i) if $|k-\ell| \geq 2$, then $G_{{\mathbf j}_0 k}\cap G_{{\mathbf j}_0\ell}=\emptyset$;

(ii) if $|k-\ell|=1$, then $G_{{\mathbf j}_0 k}\cap G_{{\mathbf j}_0\ell}$ is a line segment if and only if $$|\varepsilon|<p^{n+1},$$ and is a single point if and only if $|\varepsilon|=p^{n+1}$. (see Figure \ref{fig1})
\end{Prop}

\begin{proof}
In view of (\ref{eq2.2}) and (\ref{eq2.3}), we obtain that
\begin{equation}\label{eq2.4}
G_{{\mathbf j}_0 k}=\left\{ \left[\begin{array}{c}
p({\mathbf i})+b({\mathbf j}_0 k{\mathbf j})\\
p({\mathbf j}_0 k{\mathbf j})
\end{array}\right]: {\mathbf i},{\mathbf j}\in {\mathcal I} \right\}
\end{equation}
From the expression of the $y$-coordinate, $G_{{\mathbf j}_0 k}$ is a part of $T_\varepsilon$ between the horizontal lines $y=\sum_{t=1}^{n}\frac{j_{t}}{p^t}+\frac{k}{p^{n+1}}-\frac{1}{2p^{n+1}}$ and
$y=\sum_{t=1}^{n}\frac{j_{t}}{p^t}+\frac{k+1}{p^{n+1}}-\frac{1}{2p^{n+1}}$ . Hence the part (i) follows.

From (\ref{eq2.3}), we have
\begin{eqnarray}\label{equ.translation}
G_{{\mathbf j}_0} &=& \bigcup_{i_1,\dots,i_n}\left(A^{-n}T_\varepsilon+\left[\begin{array}{c}
p({\mathbf i_0})+b({\mathbf j}_0)\\
p({\mathbf j}_0)
\end{array}\right]\right) \nonumber \\
&=&\bigcup_{i_1,\dots,i_n}\left(A^{-n}T_\varepsilon+\left[\begin{array}{c}
p({\mathbf i_0})\\
0
\end{array}\right]\right)+\left[\begin{array}{c}
b({\mathbf j}_0)\\
p({\mathbf j}_0)
\end{array}\right] \nonumber \\
&=& G_{\mathbf 0}+\left[\begin{array}{c}
b({\mathbf j}_0)\\
p({\mathbf j}_0)
\end{array}\right].
\end{eqnarray}
That is, every $G_{{\mathbf j}_0}$ is a translation of $G_{{\mathbf 0}}$. Hence, to prove the part (ii), we only need to show the cases that $G_{{\mathbf 0}0}\cap G_{{\mathbf 0}1}$ and $G_{{\mathbf 0}0}\cap G_{{\mathbf 0}(-1)}$, as other situations are  their translations. Since  $G_{{\mathbf 0}1}$ and  $G_{{\mathbf 0}(-1)}$ are symmetric with respect to $x$-axis,    it suffices to consider $G_{{\mathbf 0}0}\cap G_{{\mathbf 0}1}$. By making use of (\ref{eq2.4}),
\begin{eqnarray*}
&& G_{{\mathbf 0}0}=\left\{ \left[\begin{array}{c}
p({\mathbf i})+b({\mathbf 0}0{\mathbf j})\\
p({\mathbf 0}0{\mathbf j})
\end{array}\right]: {\mathbf i},{\mathbf j}\in {\mathcal I} \right\} \\
&& G_{{\mathbf 0}1}=\left\{ \left[\begin{array}{c}
p({\mathbf i})+b({\mathbf 0}1{\mathbf j})\\
p({\mathbf 0}1{\mathbf j})
\end{array}\right]: {\mathbf i},{\mathbf j}\in {\mathcal I} \right\}
\end{eqnarray*}
From the proof of part (i), we know that the intersection of  $G_{{\mathbf 0}0}\cap G_{{\mathbf 0}1}$ has a unique $y$-coordinate $\frac{1}{2p^{n+1}}$. By the expression of $G_{{\mathbf 0}0}$, all digits in ${\mathbf j}$ are $m$. Since $b_i=0$ if $i$ is even; $b_i=\varepsilon$ if $i$ is odd. The $x$-coordinate of the element with $y$-coordinate of $\frac{1}{2p^{n+1}}$ in $G_{{\mathbf 0}0}$ is as follows:

If $m$ is even, $b_m=0$, then $b({\mathbf 0}0{\mathbf j})=0$. Hence the $x$-coordinate is $x=p({\mathbf i})$. Since $\{p({\mathbf i}): {\mathbf i}\in {\mathcal I}\}=[-\frac{1}{2},\frac{1}{2}]$, the $x$'s form a unit interval
$$I_1=[-\frac{1}{2},\frac{1}{2}].$$

If $m$ is odd, $b_m=\varepsilon$, then $b({\mathbf 0}0{\mathbf j})=\sum_{t=n+2}^\infty\frac{\varepsilon}{p^t}=\frac{\varepsilon}{p^{n+1}(p-1)}$. Hence the $x$-coordinate is $x=p({\mathbf i})+\frac{\varepsilon}{p^{n+1}(p-1)}$. Such $x$'s form a unit interval $$I_1=[\frac{\varepsilon}{p^{n+1}(p-1)}-\frac{1}{2},\frac{\varepsilon}{p^{n+1}(p-1)}+\frac{1}{2}].$$
In both two cases, we denote $I_1=[\alpha, \alpha+1]$.

Similarly, by the expression of $G_{{\mathbf 0}1}$, all digits in ${\mathbf j}$ are $-m$. Then the $x$-coordinate of the element with $y$-coordinate of $\frac{1}{2p^{n+1}}$ in  $G_{{\mathbf 0}1}$  is of the following form:

If $m$ is even, $x'=\frac{\varepsilon}{p^{n+1}}+ p({\mathbf i})$, which determines a unit interval $$I_2=[\frac{\varepsilon}{p^{n+1}}-\frac{1}{2},\frac{\varepsilon}{p^{n+1}}+\frac{1}{2}].$$

If $m$ is odd, $x'=\frac{\varepsilon}{p^{n+1}(p-1)}+\frac{\varepsilon}{p^{n+1}}+ p({\mathbf i})$, which determines a unit interval $$I_2=[\frac{\varepsilon}{p^{n+1}(p-1)}+\frac{\varepsilon}{p^{n+1}}-\frac{1}{2},\frac{\varepsilon}{p^{n+1}(p-1)}+\frac{\varepsilon}{p^{n+1}}+\frac{1}{2}].$$
In both two cases, we denote $I_2=[\beta,\beta+1]$.

It follows that if $G_{{\mathbf 0}0}\cap G_{{\mathbf 0}1}\ne \emptyset$, then the $y$-coordinate of the intersection is  $\frac{1}{2p^{n+1}}$  and the $x$-coordinate of the intersection is $I_1\cap I_2$. Note that  $I_1\cap I_2$ is an empty set when $|\alpha-\beta|=\frac{\varepsilon}{p^{n+1}}>1$; a single point when $|\alpha-\beta|=\frac{\varepsilon}{p^{n+1}}=1$; and an interval when  $|\alpha-\beta|=\frac{\varepsilon}{p^{n+1}}<1$. Therefore the  part (ii) is proved.
\end{proof}

\begin{Prop}\label{prop2}
Let ${\mathbf j}_0=j_1j_2\cdots j_n\in \{0,\pm1,\dots, \pm m\}^n$ for some $n\geq 1$, then $G_{{\mathbf j}_0}$  is a self-similar tile. Moreover, $G_{{\mathbf j}_0}$ is disk-like if and only if $|\varepsilon|<p^{n+1}$.
\end{Prop}

\begin{proof}
Since $G_{{\mathbf j}_0}$ is a translation of $G_{{\mathbf 0}}$ by \eqref{equ.translation}, we only show the case for $G_{{\mathbf 0}}$.
By (\ref{eq2.3}), we see that
\begin{eqnarray*}
G_{{\mathbf 0}j}&=& \bigcup_{i_1,\dots, i_{n+1}}\left(A^{-n-1}T_\varepsilon+\left[\begin{array}{c}
p(i_1\cdots i_{n+1})+b({\mathbf 0}j)\\
p({\mathbf 0}j)
\end{array}\right]\right) \\
&=& \bigcup_{i_1,\dots, i_{n+1}}A^{-1}\left(A^{-{n}}T_\varepsilon+ A\left[\begin{array}{c}
p(i_1\cdots i_{n+1})+b({\mathbf 0}j)\\
p({\mathbf 0}j)
\end{array}\right]\right) \\
&=& \bigcup_{i_1,\dots, i_{n+1}}A^{-1}\left(A^{-{n}}T_\varepsilon+\left[\begin{array}{c}
p(i_2\cdots i_{n+1})+i_1+\frac{b_j}{p^n}\\
\frac{j}{p^n}
\end{array}\right]\right) \\
&=& \bigcup_{i_1}A^{-1}\left(\bigcup_{i_2,\dots,i_{n+1}}\left(A^{-{n}}T_\varepsilon+\left[\begin{array}{c}
p(i_2\cdots i_{n+1})\\
0
\end{array}\right]\right)+\left[\begin{array}{c}
i_1+\frac{b_j}{p^n}\\
\frac{j}{p^n}
\end{array}\right]\right) \\
&=& \bigcup_{i_1}A^{-1}\left(G_{\mathbf 0}+\left[\begin{array}{c}
i_1+\frac{b_j}{p^n}\\
\frac{j}{p^n}
\end{array}\right]\right)
\end{eqnarray*}
From $G_{\mathbf 0}=\bigcup_{j=-m}^{m}G_{{\mathbf 0}j}$, it follows that
\begin{equation}\label{eq2.5}
G_{\mathbf 0}=A^{-1}(G_{\mathbf 0}+{\mathcal D}')
\end{equation}
where
$${\mathcal D}'=\left\{\left[\begin{array}{c}i+\frac{b_j}{p^n}\\ \frac{j}{p^n}\end{array}\right]: i,j\in\{0,\pm1,\dots,\pm m\}\right\}.$$
Let  $U=\left[\begin{array}{cc}
1& 0 \\
0 & p^n
\end{array}\right]$. Then $UG_{\mathbf 0}=A^{-1}(UG_{\mathbf 0}+U{\mathcal D}')$, where $$U{\mathcal D}'=\left\{\left[\begin{array}{c}i+\frac{b_j}{p^n}\\j\end{array}\right]: i,j\in\{0,\pm1,\dots,\pm m\}\right\}.$$ 
By Lemma \ref{lem2.1}, $UG_{\mathbf 0}$ is a self-similar tile, so is $G_{\mathbf 0}$. Moreover,  ${\mathcal J}'=\{(r+\ell_s,\frac{s}{p^n})^t: r,s\in\mathbb{Z}\}$ is a tiling set of $G_{\mathbf 0}$ for any sequence $\{\ell_s\}_{s\in\mathbb{Z}}\subset\mathbb{R}$.

According to \eqref{eq2.5}, we can write the corresponding IFS of $G_{\mathbf 0}$ as 
$$f_{i,j}'\left(\left[\begin{array}{c}x\\y\end{array}\right]\right)=A^{-1}\left(\left[\begin{array}{c}x\\y\end{array}\right]+\left[\begin{array}{c}i+\frac{b_j}{p^n}\\ \frac{j}{p^n}\end{array}\right]\right),\quad i,j\in\{0,\pm1,\dots,\pm m\}.$$
We will verify the disk-likeness of $G_{\mathbf 0}$ through the following claims.

(i) We claim that $f'_{i,j}(G_{\mathbf 0})\cap f'_{i+1,j}(G_{\mathbf 0})$ contains an interior point of $G_{\mathbf 0}$ for any $i,j$. It suffices to show that $G_{\mathbf 0}\cap (G_{\mathbf 0}+(1,0)^t)$ contains a point of $(G_{\mathbf 0}\cup (G_{\mathbf 0}+(1,0)^t))^\circ$. For ${\mathbf i}, {\mathbf j}\in {\mathcal I}$, we have $p(\mathbf i)\in [-\frac{1}{2},\frac{1}{2}]$ and $\frac{p(\mathbf j)}{p^n}\in [-\frac{1}{2p^n},\frac{1}{2p^n}]$. Fix a point $y_0\in(-\frac{1}{2p^n},\frac{1}{2p^n})$ so that there exists a unique $\mathbf j\in {\mathcal I}$ such that $\frac{p(\mathbf j)}{p^n}=y_0$. Let $$x_0=\frac{b(\mathbf j)}{p^n}+\frac{1}{2},$$ then $(x_0,y_0)^t\in G_{\mathbf 0}\cap (G_{\mathbf 0}+(1,0)^t)$.

On the other hand, since $y_0$ is an interior point of $[-\frac{1}{2p^n},\frac{1}{2p^n}]$ and the set ${\mathcal J}'=\{(r,\frac{s}{p^n})^t: r,s\in\mathbb{Z}\}$ is a tiling set of $G_{\mathbf 0}$ (taking $\ell_s\equiv0$), then $(x_0,y_0)^t\in G_{\mathbf 0}+(r,\frac{s}{p^n})^t$ only if $s=0$,  i.e., $(x_0,y_0)^t\in G_{\mathbf 0}+(r,0)^t$. Thus, $x_0$ must be the form
$$x_0=\frac{b(\mathbf j)}{p^n}+p(\mathbf i)+r$$
for some $\mathbf i\in {\mathcal I}$.  That implies $p(\mathbf i)+r=\frac{1}{2}$, hence $r=0\ \text{or}\ 1$ as $p(\mathbf i)\in [-\frac{1}{2},\frac{1}{2}]$. Then $(x_0,y_0)^t\notin G_{\mathbf 0}+(r,0)^t$ for any $r\in {\mathbb Z}\setminus\{0,1\}$. Therefore $(x_0,y_0)^t$ must be in $(G_{\mathbf 0}\cup (G_{\mathbf 0}+(1,0)^t))^\circ$.

(ii) We claim that $G_{\mathbf 0}$ is connected if and only if $|\varepsilon|\leq p^{n+1}$. Indeed, if $G_{\mathbf 0}$ is connected,  by Proposition \ref{prop1},  then $G_{{\mathbf 0}k}\cap G_{{\mathbf 0} (k+1)}\neq\emptyset$ for  $-m\leq k\leq m-1$, which implies $|\varepsilon|\leq p^{n+1}$. On the contrary, if $|\varepsilon|\leq p^{n+1}$, then $G_{\mathbf 0k}\cap G_{\mathbf 0 (k+1)}\neq\emptyset$ for $-m\leq k\leq m-1$,  by Proposition \ref{prop1}. Thus, there exist $i_k,t_k\in\{0,\pm 1,\dots,\pm m\}$ such that $f'_{i_k,k}(G_{\mathbf 0})\cap f'_{t_k,k+1}(G_{\mathbf 0})\neq\emptyset$. That together with (i), can help us select a finite sequence $\{S_j\}_{j=1}^N$ from $\{f'_{i,j}\}_{i,j}$  in the following zigzag order:
\begin{eqnarray*}
& & f'_{-m,-m},f'_{-m+1,-m},\dots,f'_{m,-m},f'_{m-1,-m},\dots,f'_{i_{-m},-m},f'_{t_{-m},-m+1},f'_{t_{-m}+1,-m+1},\dots,\\
& & f'_{m,-m+1},f'_{m-1,-m+1},\dots,f'_{-m,-m+1},f'_{-m+1,-m+1},\dots,f'_{i_{-m+1},-m+1},f'_{t_{-m+1},-m+2},\dots,\\
& & f'_{i_{m-1},m-1},f'_{t_{m-1},m},\dots,f'_{m,m},f'_{m-1,m},\dots,f'_{-m,m}.
\end{eqnarray*}
Then each $f'_{i,j}$ appears at least once in the sequence  $\{S_j\}_{j=1}^N$ and $$S_j(G_{\mathbf 0})\cap S_{j+1}(G_{\mathbf 0})\ne\emptyset, \quad  \forall \ 1\le j\le N-1.$$ 
Hence $G_{\mathbf 0}$ is connected by Lemma \ref{lem2.2}.

\begin{figure}
  \centering
   \subfigure[$\varepsilon=2$]{
  \includegraphics[width=4cm, height=3cm]{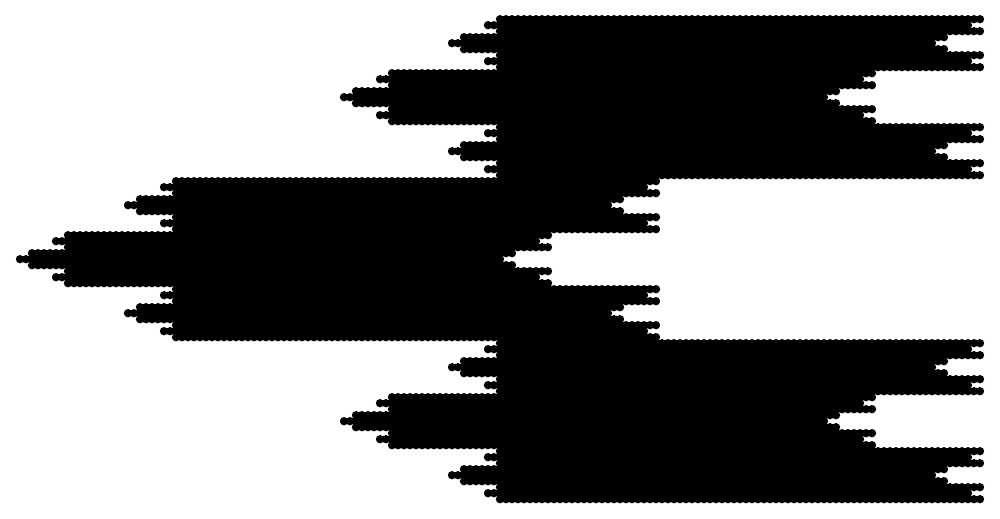}
 } \quad
 \subfigure[$\varepsilon=3$]{
  \includegraphics[width=4cm, height=3cm]{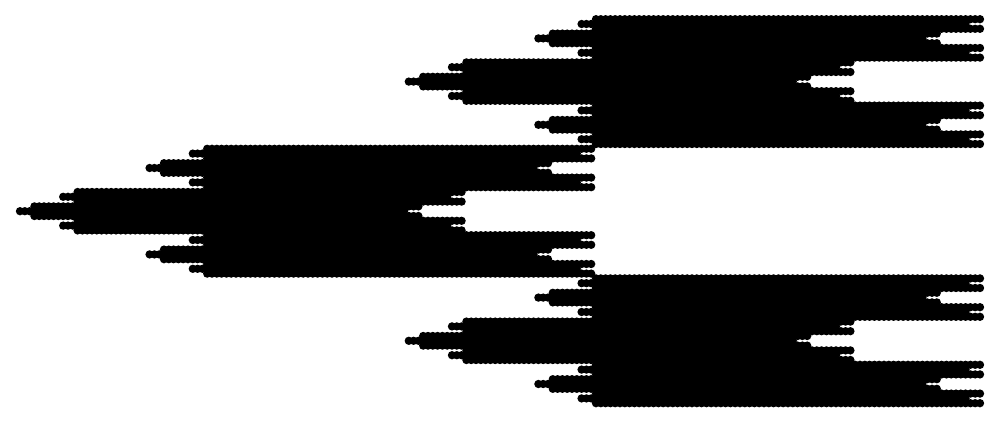}
 } \quad
 \subfigure[$\varepsilon=4$]{
   \includegraphics[width=4cm, height=3cm]{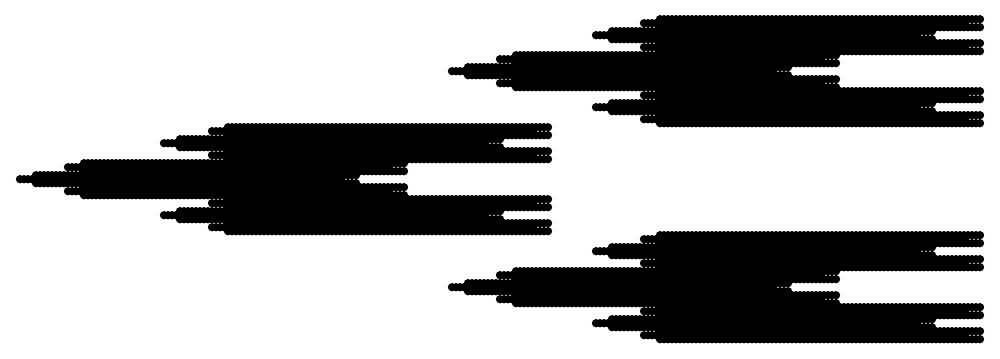}
 }\\
 \subfigure[$\varepsilon=8$]{
  \includegraphics[width=4cm, height=3cm]{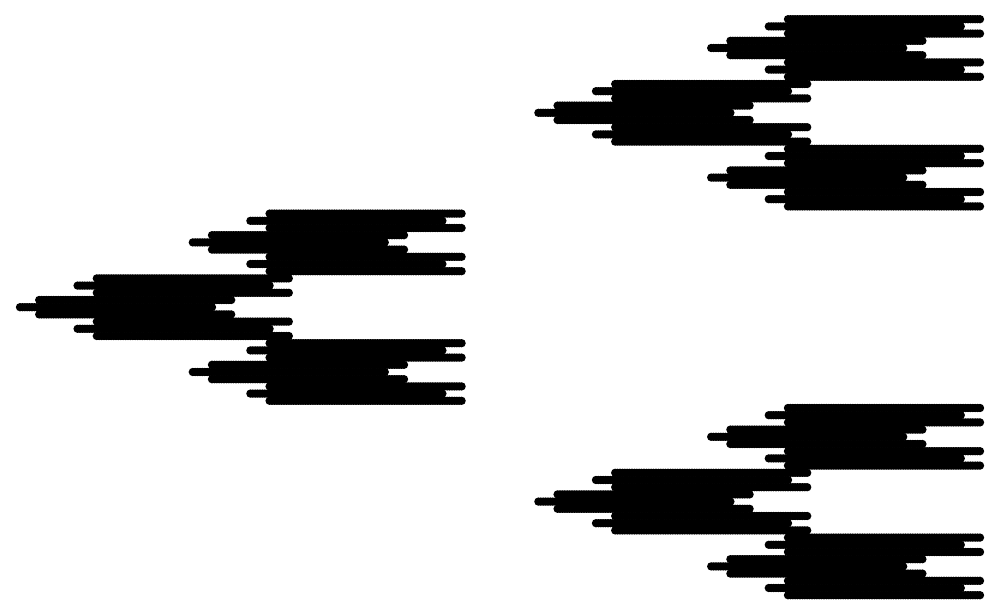}
 } \quad
 \subfigure[$\varepsilon=9$]{
   \includegraphics[width=4cm, height=3cm]{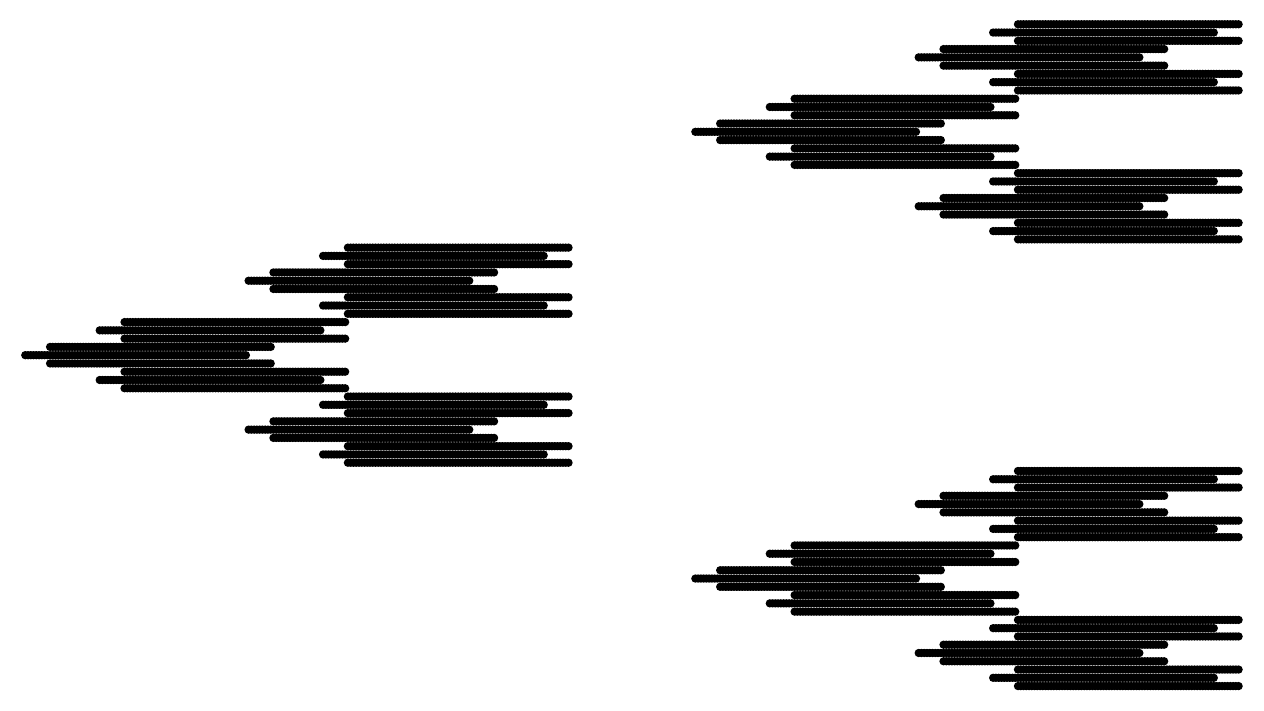}
 }\quad
 \subfigure[$\varepsilon=10$]{
   \includegraphics[width=4cm, height=3cm]{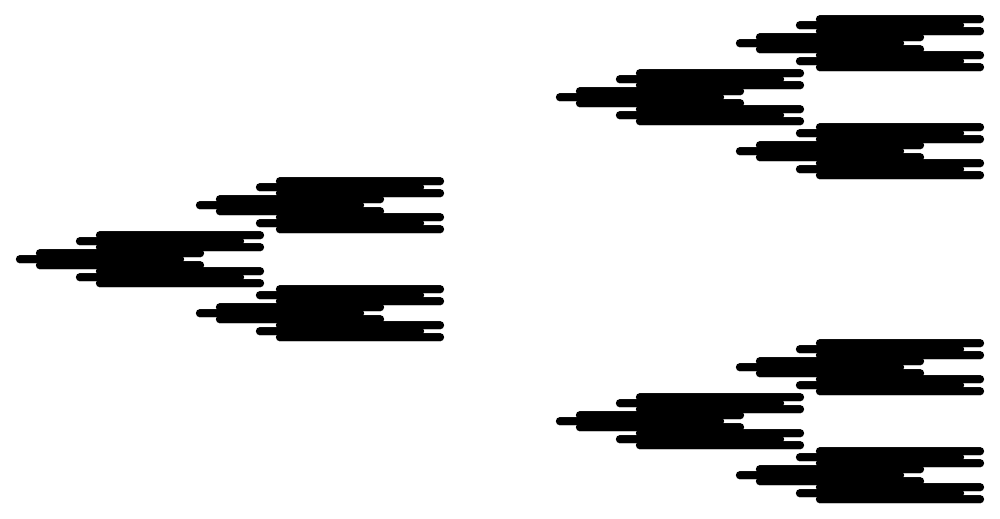}
 }
 \caption{An illustration of Theorem \ref{mainthm} by taking $p=3$.}\label{fig1}
\end{figure}

(iii) We claim that $(G_{\mathbf 0})^\circ$ is connected if $|\varepsilon|< p^{n+1}$. For $-m\leq k\leq m-1$, if $|\varepsilon|< p^{n+1}$, then
there exist $i_k,t_k$ such that $f'_{i_k,k}(G_{\mathbf 0})\cap f'_{t_k,k+1}(G_{\mathbf 0})$ is a horizontal line segment by Proposition \ref{prop1}. Suppose $\overline{z}=f'_{i_k,k}(\tilde{z})$ be the mid-point of the line segment. Since
$$f'_{i_k,k}(G_{\mathbf 0})=A^{-1}\left(G_{\mathbf 0}+\left[\begin{array}{c}
i_k+\frac{b_k}{p^n}\\
\frac{k}{p^n}
\end{array}\right]\right),$$
$$f'_{t_k,k+1}(G_{\mathbf 0})=A^{-1}\left(G_{\mathbf 0}+\left[\begin{array}{c}
i_k+\frac{b_k}{p^n}\\
\frac{k}{p^n}
\end{array}\right]+\left[\begin{array}{c}
c\\
\frac{1}{p^n}
\end{array}\right]\right),$$
where $c=t_k-i_k+\frac{b_{k+1}-b_k}{p^n}$. Then $G_{\mathbf 0}\cap (G_{\mathbf 0}+(c,\frac{1}{p^n})^t)$ is a horizontal line segment with positive length and $\tilde{z}$ is its mid-point. Note that the top and bottom sides of $G_{\mathbf 0}$ are horizontal line segments with length one and the height of $G_{\mathbf 0}$ is $\frac{1}{p^n}$, it is easy to verify that
$$\tilde{z}\in G_{\mathbf 0}+(r+\ell_s,\frac{s}{p^n})^t \ \  \text{if and only if } \ (r+\ell_s,\frac{s}{p^n})^t=(0,0)^t \ \text{or} \ (c,\frac{1}{p^n})^t.$$ Then $\tilde{z}$ is an interior point of $G_{\mathbf 0}\cup (G_{\mathbf 0}+(c,\frac{1}{p^n})^t)$, this means $\overline{z}$ is an interior point of $f'_{i_k,k}(G_{\mathbf 0})\cup f'_{t_k,k+1}(G_{\mathbf 0})$. Thus, $f'_{i_k,k}(G_{\mathbf 0})\cap f'_{t_k,k+1}(G_{\mathbf 0})$ contains an interior point of $G_{\mathbf 0}$. Hence $(G_{\mathbf 0})^\circ$ is connected by Lemma \ref{lem2.2}.

Now we prove the second part of the proposition.  Suppose $|\varepsilon|<p^{n+1}$, then claim (iii) implies $(G_{\mathbf 0})^\circ$ is connected, which yields the disk-likeness of $G_{\mathbf 0}$ by a theorem of Luo {\it et al.} \cite{LRT}. Suppose $G_{\mathbf 0}$ is disk-like, then $G_{\mathbf 0}$ and $(G_{\mathbf 0})^\circ$ are both connected, hence we have  $|\varepsilon|\leq p^{n+1}$ by claim (ii). If the equality holds, then $G_{\mathbf 0k}\cap G_{\mathbf 0(k+1)}$ contains only one point for any $k$, and  $G_{\mathbf 0k}\cap G_{\mathbf 0\ell}=\emptyset$ for $|k-\ell|\geq 2$ by Proposition \ref{prop1}. Since $G_{\mathbf 0}=\bigcup_{k=-m}^{m}G_{{\mathbf 0}k}$, $G_{\mathbf 0}$ can be divided into $p$ parts, and every two adjacent parts intersect at one common point. Hence the intersection point must be at the boundary of $G_{\mathbf 0}$. That is,  $(G_{\mathbf 0})^\circ$ is not connected, yielding a contradiction.  Therefore, $G_{\mathbf 0}$ is disk-like if and only if $|\varepsilon|<p^{n+1}$.
\end{proof}

\bigskip

\noindent{\it \textbf{Proof of Theorem \ref{mainthm}:}}
(i) If $|\varepsilon|<|p|$, by using the same argument as in the proof of Proposition \ref{prop2}, we can prove that $T_\varepsilon$ is disk-like (see also Theorem 3.1 of \cite{DeLa}); If $|\varepsilon|<|p|^{n+1}$ for $n \geq 1$, by Proposition \ref{prop2}, then $G_{j_1\cdots j_n}$ are disk-like tiles. Moreover, from Proposition \ref{prop1} and the assumption $|\varepsilon|\ge |p|^n$, it follows that the tiles $G_{j_1\cdots j_n}$ are either disjoint or meet each other at a single point (see Figure \ref{fig1}). Therefore, we conclude with proving (ii).
\hfill $\Box$

\section{\bf Quasi-periodic tiling}
A tiling $T+{\mathcal J}$ of ${\mathbb R}^n$ with a tile $T$  is called a {\it self-replicating tiling} with a matrix $B$ if for each $t\in {\mathcal J}$ there exists a finite set $J(t)\subset {\mathcal J}$ such that $$B(T+t)=\bigcup_{t'\in J(t)}(T+t').$$
Moreover, we say a tiling $T+{\mathcal J}$ is a {\it quasi-periodic tiling} if the following two properties hold:

\begin{enumerate}
\item  Local isomorphism property: For any finite set $\Sigma\subset {\mathcal J}$, there exists a constant $R>0$ such that every ball of radius $R$ in ${\mathbb R}^n$ contains a translating copy $\Sigma+t$ of $\Sigma$ such that $\Sigma+t\subset {\mathcal J}$.

\item Local finiteness property: For any $k\ge 1$ and $r>0$, there  are finitely many translation-inequivalent arrangements of $k$ points in $\mathcal J$ which lie in some ball of radius $r$.
\end{enumerate}

Under the assumption of Theorem \ref{mainthm}, we let ${\mathcal D}_{\varepsilon,k}={\mathcal D}_\varepsilon+ A{\mathcal D}_\varepsilon+\cdots+ A^{k-1}{\mathcal D}_\varepsilon$ and ${\mathcal D}_{\varepsilon,\infty}=\bigcup_{k=1}^\infty{\mathcal D}_{\varepsilon,k}$. For $j\in \{0,\pm1,\dots, \pm m\}$, denote by
\begin{equation*}
\widehat{j}=\left\{
\begin{array}{ll}
1 & j \ \text{is odd} \\
0 &  j \ \text{is even}.
\end{array} \right.
\end{equation*}

Then it can be easily seen that
\begin{align*}
{\mathcal D}_{\varepsilon,\infty} &= \left\{\left[\begin{array}{c}
\sum_{n=0}^\infty i_n p^n+ \varepsilon\sum_{n=0}^\infty \widehat{j}_n p^n\\
\sum_{n=0}^\infty j_n p^n
\end{array}\right]: i_n, j_n\in\{0,\pm1,\dots,\pm m\}\right\} \\
&=  \left\{\left[\begin{array}{c}
m_2+\widehat{m}_1\varepsilon\\
m_1
\end{array}\right]: m_1,m_2\in{\mathbb Z}, \widehat{m}_1=\sum_{n=0}^\infty \widehat{j}_n p^n \text{ when } m_1=\sum_{n=0}^\infty j_n p^n\right\}.
\end{align*}

\begin{Prop}
 $T_\varepsilon+{\mathcal D}_{\varepsilon,\infty}$ is a self-replicating tiling.
\end{Prop}

\begin{proof}
By Lemma \ref{lem2.1}, obviously ${\mathcal D}_{\varepsilon,\infty}$ is a tiling set of ${\mathbb R}^2$. For any $t\in {\mathcal D}_{\varepsilon,\infty}$, we let $J(t)=At+{\mathcal D}_\varepsilon$. Since $AT_\varepsilon=T_\varepsilon+{\mathcal D}_\varepsilon$, we have $AT_\varepsilon+At=T_\varepsilon+{\mathcal D}_\varepsilon+At$. Hence $$A(T_\varepsilon+t)=\bigcup_{t'\in J(t)}(T_\varepsilon+t').$$ Therefore $T_\varepsilon+{\mathcal D}_{\varepsilon,\infty}$ is a self-replicating tiling by letting $B=A$.
\end{proof}

\begin{theorem}
$T_\varepsilon+{\mathcal D}_{\varepsilon,\infty}$ is a quasi-periodic tiling if and only if $\varepsilon$ is a rational number.
\end{theorem}

\begin{proof}
By the definition, we first show the local isomorphism property holds. Let $\Sigma$ be a finite subset of ${\mathcal D}_{\varepsilon,\infty}$, then there exists an integer $\ell$ such that $\Sigma\subset {\mathcal D}_{\varepsilon,\ell}$. Let $$R=\text{diam}(A^\ell T_\varepsilon)=\text{diam}(\bigcup_{d\in {\mathcal D}_{\varepsilon,\ell}}(T_\varepsilon+d)).$$ We claim that every ball $B_x(3R)$ with center $x\in {\mathbb R}^2$ and radius $3R$, contains a translating copy of $\Sigma$. Indeed, for the $x$, there exists a $d_x\in {\mathcal D}_{\varepsilon,\infty}$ such that $x\in T_\varepsilon+d_x$ by the tiling property. If $d_x\notin {\mathcal D}_{\varepsilon,\ell}$, then there is a larger integer $\ell'>\ell$ such that
$$d_x\in \sum_{j=0}^{\ell-1}A^jd_{i_j}+\sum_{j=\ell}^{\ell'}A^jd_{i_j}, \quad \text{where } d_{i_j}\in {\mathcal D}_\varepsilon.$$ Let $d_x'=\sum_{j=\ell}^{\ell'}A^jd_{i_j}$, then $d_x-d_x'\in {\mathcal D}_{\varepsilon,\ell}$. By using $0\in T_\varepsilon$ and the triangle inequality,  for $d\in {\mathcal D}_{\varepsilon,\ell}$, we have, $$\|d_x'+d-x\|\le  \|x-d_x\|+\|d_x-d_x'\|+\|d\|\le 3R.$$
Thus $d_x'+{\mathcal D}_{\varepsilon,\ell}\subset B_x(3R)$, yielding $d_x'+\Sigma\subset B_x(3R)$. If $d_x\in {\mathcal D}_{\varepsilon,\ell}$, the above is still true as $d_x'=0$.

Now we show $T_\varepsilon+{\mathcal D}_{\varepsilon,\infty}$ has the local finiteness property if and only if $\varepsilon$ is rational, then the desired result follows.

Suppose $\varepsilon$ is irrational. Consider $m_1=m(1+p+\cdots+p^{k-1})=\frac{1}{2}(p^k-1), \   n_1=(-m)(1+p+\cdots+p^{k-1})+ 1\cdot p^k=\frac{1}{2}(p^k+1)$. If $m$ is an even integer, then $\widehat{m}_1=0, \widehat{n}_1=p^k$. We have $$t_k:=\left[\begin{array}{c}
m_2\\
m_1
\end{array}\right], \   t_k':=\left[\begin{array}{c}
n_2+\varepsilon p^k\\
n_1
\end{array}\right]\in {\mathcal D}_{\varepsilon,\infty},  \  \forall  \   m_2, n_2\in {\mathbb Z}.$$
If $m$ is an odd integer, then $\widehat{m}_1=1+p+\cdots+p^{k-1}=\frac{1-p^k}{1-p},\ \widehat{n}_1=1+p+\cdots+p^{k-1}+p^k=\frac{1-p^k}{1-p}+p^k$. We have $$t_k:=\left[\begin{array}{c}
m_2+\frac{\varepsilon(1-p^k)}{1-p}\\
m_1
\end{array}\right], \   t_k':=\left[\begin{array}{c}
n_2+\varepsilon p^k+\frac{\varepsilon(1-p^k)}{1-p}\\
n_1
\end{array}\right]\in {\mathcal D}_{\varepsilon,\infty},  \  \forall  \  m_2, n_2\in {\mathbb Z}.$$
In any case, by letting $m_2$ be the integral part of $n_2+\varepsilon p^k$, we always get  $\|t_k-t_k'\|\le \sqrt{2}$. Thus there are infinitely many translation-inequivalent arrangements of $t_k$ and $t_k'$.

On the other hand, suppose $\varepsilon=\frac ab$ is a rational number where $a,b$ are co-prime numbers. Given an integer $c>0$, and any two elements
$$\left[\begin{array}{c}
m_2+\widehat{m}_1\varepsilon\\
m_1
\end{array}\right], \  \left[\begin{array}{c}
n_2+\widehat{n}_1\varepsilon\\
n_1
\end{array}\right]\in {\mathcal D}_{\varepsilon,\infty}.$$
$|m_1-m_2|\le c$ can admit at most $2c+1$ choices for $m_1-m_2$; while $b|(m_2+\widehat{m}_1\varepsilon)-(n_2+\widehat{n}_1\varepsilon)|=|b(m_2-n_2)+a(\widehat{m}_1-\widehat{n}_1)|\in {\mathbb Z}$, hence  $|(m_2+\widehat{m}_1\varepsilon)-(n_2+\widehat{n}_1\varepsilon)|\le c$ can admit at most $b(2c+1)$ choices for $(m_2+\widehat{m}_1\varepsilon)-(n_2+\widehat{n}_1\varepsilon)$. Thus the number of possible differences of any two points of ${\mathcal D}_{\varepsilon,\infty}$ in a square $S$ with edge length $c$ is finite. Inductively, we have the same conclusion for $k$ points in $S$. Therefore, $T_\varepsilon+{\mathcal D}_{\varepsilon,\infty}$ has the local finiteness property.
\end{proof}

\section{\bf Self-similar sets}
Let $T_{\varepsilon}:=T(A, {\mathcal D}_\varepsilon)$ be the associated self-similar set in Theorem \ref{mainthm3}. We only need to consider the case that $p>0$. Under the assumption of Theorem \ref{mainthm3}, we can write the IFS as follows
\begin{equation}\label{eq.IFS}
f_{i,j}\left(\left[\begin{array}{c}
x\\
y
\end{array}\right]\right)=A^{-1}\left(\left[\begin{array}{c}
x\\
y
\end{array}\right]+\left[\begin{array}{c}
i+a_{ij}\\
j+a_{ij}
\end{array}\right]\right)
\end{equation}
where $i,j\in\{0,1,\dots,p-1\}, \  a_{ij}=\delta_{ij}\varepsilon$. Denote by $t_{i,j}$  the fixed point of  the contraction $f_{i,j}$. It is easy to compute that
$$t_{0,0}=\left[\begin{array}{c}\frac{\varepsilon}{p-1}\\ \frac{\varepsilon}{p-1}\end{array}\right],t_{1,0}=\left[\begin{array}{c}\frac{1}{p-1}\\ 0\end{array}\right],t_{0,1}=\left[\begin{array}{c}0\\\frac{1}{p-1}\end{array}\right],t_{1,p-1}=\left[\begin{array}{c}\frac{1}{p-1}\\1\end{array}\right],$$
$$t_{0,p-1}=\left[\begin{array}{c}0\\1\end{array}\right],t_{p-1,0}=\left[\begin{array}{c}1\\0\end{array}\right],t_{p-1,1}=\left[\begin{array}{c}1\\\frac{1}{p-1}\end{array}\right],
t_{p-1,p-1}=\left[\begin{array}{c}1+\frac{\varepsilon}{p-1}\\1+\frac{\varepsilon}{p-1}\end{array}\right].$$

The following simple result is very useful for proving the connectedness of $T_{\varepsilon}$.

\begin{Prop}\label{prop4}
(i) for $i\neq j$ and $i+1\neq j$, we have $f_{i,j}(T_{\varepsilon})\cap f_{i+1,j}(T_{\varepsilon})\neq\emptyset;$

(ii) for  $j\neq i$ and $j+1\neq i$, we have $f_{i,j}(T_{\varepsilon})\cap f_{i,j+1}(T_{\varepsilon})\neq\emptyset$;

(iii) for $i=0,1,\dots,p-2$, we have $f_{i+1,i}(T_{\varepsilon})\cap f_{i,i+1}(T_{\varepsilon})\neq\emptyset$ and  $f_{i,i}(T_{\varepsilon})\cap f_{i+1,i+1}(T_{\varepsilon})\neq\emptyset$.
\end{Prop}

\begin{proof}
(i) For $i\neq j\ \text{and}\ i+1\neq j,$ we see that
$$f_{i,j}\left(\left[\begin{array}{c}x\\y\end{array}\right]\right)=\left[\begin{array}{c}\frac{x+i}{p}\\\frac{y+j}{p}\end{array}\right]\ \text{and}\ \
f_{i+1,j}\left(\left[\begin{array}{c}x\\y\end{array}\right]\right)=\left[\begin{array}{c}\frac{x+i+1}{p}\\\frac{y+j}{p}\end{array}\right].$$
Since $$f_{i,j}(t_{p-1,1})=\left[\begin{array}{c}\frac{i+1}p\\\frac{1}{p(p-1)}+\frac{j}{p}\end{array}\right]=f_{i+1,j}(t_{0,1}),$$  we have $f_{i,j}(T_{\varepsilon})\cap f_{i+1,j}(T_{\varepsilon})\neq\emptyset$. By symmetry of $T_{\varepsilon}$ (with respect to line $y=x$), (ii) also holds.

(iii) For $i=0,1,\dots,p-2,$ we see that
$$f_{i+1,i}\left(\left[\begin{array}{c}x\\y\end{array}\right]\right)=\left[\begin{array}{c}\frac{x+i+1}p\\\frac{y+i}p\end{array}\right]\ \text{and}\ \
f_{i,i+1}\left(\left[\begin{array}{c}x\\y\end{array}\right]\right)=\left[\begin{array}{c}\frac{x+i}p\\\frac{y+i+1}p\end{array}\right].$$
Since $$f_{i+1,i}(t_{0,p-1})=\left[\begin{array}{c}\frac{i+1}p\\\frac{i+1}p\end{array}\right]=f_{i,i+1}(t_{p-1,0}),$$
 we have  $f_{i+1,i}(T_{\varepsilon})\cap f_{i,i+1}(T_{\varepsilon})\neq\emptyset$.

From
$$f_{i,i}\left(\left[\begin{array}{c}x\\y\end{array}\right]\right)=\left[\begin{array}{c}\frac{x+i+\varepsilon}p\\\frac{y+i+\varepsilon}p\end{array}\right],$$
it follows that $$f_{i,i}(t_{p-1,p-1})=\left[\begin{array}{c}\frac{i+1+\varepsilon}p+\frac{\varepsilon}{p(p-1)}\\ \frac{i+1+\varepsilon}p+\frac{\varepsilon}{p(p-1)}\end{array}\right]=f_{i+1,i+1}(t_{0,0}).$$
Hence  $f_{i,i}(T_{\varepsilon})\cap f_{i+1,i+1}(T_{\varepsilon})\neq\emptyset$.
\end{proof}

\bigskip

\noindent{\it \textbf{Proof of Theorem \ref{mainthm3}:}}

First we assume that $\varepsilon\geq 0$.
For the function system $\{f_{0,0},f_{1,1}, \dots, f_{p-1,p-1}\}$,
there exists a unique attractor
$$\pi=seg\left(\left[\begin{array}{c}\frac{\varepsilon}{p-1}\\ \frac{\varepsilon}{p-1}\end{array}\right],\left[\begin{array}{c}1+\frac{\varepsilon}{p-1}\\1+\frac{\varepsilon}{p-1}\end{array}\right]\right)$$
where $seg(u,v)$ denotes the straight line-segment between vector $u$ and vector $v$.

Let \[V_{1}=\bigcup_{i<j}f_{i,j}(T_{\varepsilon}),\quad  V_{2}=\bigcup_{i=j}f_{i,j}(T_{\varepsilon}),\quad  V_{3}=\bigcup_{i>j}f_{i,j}(T_{\varepsilon}).\]
Then $T_{\varepsilon}=V_{1}\cup V_{2}\cup V_{3} \;\text{and}\;\pi\subset V_{2}.$
It is easy to see that $T_{\varepsilon}$ is located between two lines:
$$L_{1}:=\left\{\left[\begin{array}{c}x\\y\end{array}\right]\in\mathbb{R}^{2}: y=x+1\right\},\ \ L_{2}:=\left\{\left[\begin{array}{c}x\\y\end{array}\right]\in\mathbb{R}^{2}: y=x-1\right\}.$$
$V_{2}$ is located between two lines:
$$L_{3}:=\left\{\left[\begin{array}{c}x\\y\end{array}\right]\in\mathbb{R}^{2}: y=x+\frac{1}{p}\right\},\ \ L_{4}:=\left\{\left[\begin{array}{c}x\\y\end{array}\right]\in\mathbb{R}^{2}: y=x-\frac{1}{p}\right\}.$$
$V_{1},\;V_{3}$ are separated by a line:
$$L_{5}:=\left\{\left[\begin{array}{c}x\\y\end{array}\right]\in\mathbb{R}^{2}: y=x\right\}.$$
Obviously, the segment $\pi\subset L_{5}$. (see Figure \ref{fig.2}(a))

By Proposition \ref{prop4}, for any distinct pairs $(i,j),\ (i',j')$ with $i\neq j$ and $i'\neq j'$, there exists a sequence $(i,j)=(i_1,j_1),(i_2,j_2),\dots,(i_n,j_n)=(i',j')$ with $i_k\neq j_k,\ 1\leq k\leq n$ so that  $f_{i_k,j_k}(T_{\varepsilon})\cap f_{i_{k+1},j_{k+1}}(T_{\varepsilon})\neq\emptyset$ for all $1\leq k<n.$ On the other hand, for any $i\neq j\in\{0,1,\dots,p-1\}$ there exists a sequence $i=j_1,j_2,\dots,j_m=j$ of indices in $\{0,1,\dots,p-1\}$ so that $f_{j_k,j_k}(T_{\varepsilon})\cap f_{j_{k+1},j_{k+1}}(T_{\varepsilon})\neq\emptyset$ for all $1\leq k<m$.
Hence, by Lemma \ref{lem2.2} and symmetry of $T_{\varepsilon}$ (with respect to the line $L_5$), $T_{\varepsilon}$ is connected if and only if $V_{1}\cap V_{2}\neq\emptyset$.

Since
$$f_{0,1}(t_{p-1,0})=\left[\begin{array}{c}\frac1p\\\frac1p\end{array}\right]\in f_{0,1}(T_{\varepsilon})\subset V_{1}$$ and
$$f_{p-2,p-1}(t_{p-1,0})=\left[\begin{array}{c}\frac{p-1}p\\\frac{p-1}p\end{array}\right]\in f_{p-2,p-1}(T_{\varepsilon})\subset V_{1}.$$
It follows that these two points belong to the line $L_{5}$.  We shall consider $V_1\cap V_2$ by comparing these two points with the endpoints of the segment $\pi$.

(i) If  $\varepsilon\leq\frac{(p-1)^{2}}{p}$, i.e., $\frac{\varepsilon}{p-1}\leq\frac{p-1}{p}$. Then $f_{p-2,p-1}(t_{p-1,0})\in\pi,$ that is $V_{1}\cap V_{2}\neq\emptyset,$
which implies $T_{\varepsilon}$ is connected.

(ii) If $\varepsilon>\frac{(p-1)^{2}}{p}$, i.e., $\frac{\varepsilon}{p-1}>\frac{p-1}{p}$.  We let a point
$$\omega:=f_{0,0}\circ f_{1,p-1}(t_{0,p-1})=\left[\begin{array}{c}\frac1{p^2}+\frac{\varepsilon}p\\\frac1p+\frac{\varepsilon}p\end{array}\right]\in f_{0,0}(T_{\varepsilon})\subset V_{2}$$ and a line
$$L_{6}:=\left\{\left[\begin{array}{c}x\\y\end{array}\right]\in\mathbb{R}^{2}: y=x+\frac{1}{p}-\frac{1}{p^2}\right\}.$$
Then $\omega\in L_6$ (see Figure \ref{fig.2}(b)).  Consider the compositions $f_{p-2,p-1}\circ f_{i+1,i}$ where $i=0,1,\dots,p-2$, we have
$$f_{p-2,p-1}\circ f_{i+1,i}(\pi)
=seg\left(\left[\begin{array}{c}\frac{\varepsilon}{p^2(p-1)}+\frac{i+1}{p^2}+\frac{p-2}{p}\\ \frac{\varepsilon}{p^2(p-1)}+\frac{i}{p^2}+\frac{p-1}{p}\end{array}\right],\left[\begin{array}{c} \frac{\varepsilon}{p^2(p-1)}+\frac{i+2}{p^2}+\frac{p-2}{p}\\ \frac{\varepsilon}{p^2(p-1)}+\frac{i+1}{p^2}+\frac{p-1}{p}\end{array}\right]\right).$$
Obviously, $f_{p-2,p-1}\circ f_{i+1,i}(\pi)\subset L_{6}$ holds for any $i$, and the right endpoint  of $f_{p-2,p-1}\circ f_{i+1,i}(\pi)$ is equal to the left endpoint of $f_{p-2,p-1}\circ f_{i+2,i+1}(\pi)$. That is,
$$\pi_{1}:=\bigcup^{p-2}_{i=0}f_{p-2,p-1}\circ f_{i+1,i}(\pi)
=seg\left(\left[\begin{array}{c}\frac{\varepsilon}{p^2(p-1)}+\frac{(p-1)^2}{p^2}\\ \frac{\varepsilon}{p^2(p-1)}+\frac{p-1}{p}\end{array}\right],\left[\begin{array}{c} \frac{\varepsilon}{p^2(p-1)}+\frac{p-1}{p}\\ \frac{\varepsilon}{p^2(p-1)}+\frac{p^2-1}{p^2}\end{array}\right]\right).$$
Then  $\pi_{1}\subset f_{p-2,p-1}(T_{\varepsilon})\subset V_{1}$ and $\pi_1\subset L_{6}$ (see Figure \ref{fig.2}(b)). Therefore, to show $V_1\cap V_2\ne\emptyset$, we only need to compare the $x$-coordinates of $\omega$ and  $\pi_{1}$.

If
$$\frac{\varepsilon}{p^2(p-1)}+\frac{(p-1)^2}{p^2}\leq\frac{\varepsilon}{p}+\frac{1}{p^2}\leq\frac{\varepsilon}{p^2(p-1)}+\frac{p-1}{p},$$
i.e., $$ \frac{p(p-1)(p-2)}{p^2-p-1}\leq\varepsilon\leq p-1,$$
then  $\omega\in\pi_{1}$. Hence  $V_{1}\cap V_{2}\neq\emptyset,$ which implies $T_{\varepsilon}$ is connected.  As $$\frac{p(p-1)(p-2)}{p^2-p-1}<\frac{(p-1)^2}{p} \ \ \text{when} \  \  p>2,$$ then $T_{\varepsilon}$ is connected if $\varepsilon\leq p-1$.

(iii) If $p-1<\varepsilon\leq\frac{(p-1)^2}{p-2}$, by using a similar argument as above, we can find a point $\theta$ in $V_2$ and a line-segment $\pi_2$ in $V_1$ so that $\theta\in\pi_2$. Let $$\theta:=f_{0,0}(t_{0,p-1})=\left[\begin{array}{c}\frac{\varepsilon}{p}\\\frac{\varepsilon+1}p\end{array}\right]\in f_{0,0}(T_{\varepsilon})\subset V_{2},$$
then $\theta\in L_{3}.$ Define a line-segment $$\pi_2:=f_{p-2,p-1}(\pi)
=seg\left(\left[\begin{array}{c}\frac{\varepsilon}{p(p-1)}+\frac{p-2}p\\ \frac{\varepsilon}{p(p-1)}+\frac{p-1}p\end{array}\right],\left[\begin{array}{c}\frac{\varepsilon}{p(p-1)}+\frac{p-1}p\\ \frac{\varepsilon}{p(p-1)}+1\end{array}\right]\right).$$
Trivially we have $\pi_2\subset f_{p-2,p-1}(T_\varepsilon)\subset V_{1}$ and $\pi_2\subset L_{3}$. (see Figure \ref{fig.2}(c))

From $p-1<\varepsilon\leq\frac{(p-1)^2}{p-2}$, it follows that
$$\frac{\varepsilon}{p(p-1)}+\frac{p-2}{p}<\frac{\varepsilon}{p}\leq\frac{\varepsilon}{p(p-1)}+\frac{p-1}{p}.$$
Then $\theta\in\pi_2$, and $V_{1}\cap V_{2}\neq\emptyset$. Hence $T_{\varepsilon}$ is connected.

\begin{figure}[h]
	\centering
	\subfigure[]{
		\includegraphics[width=4.2cm, height=4.8cm]{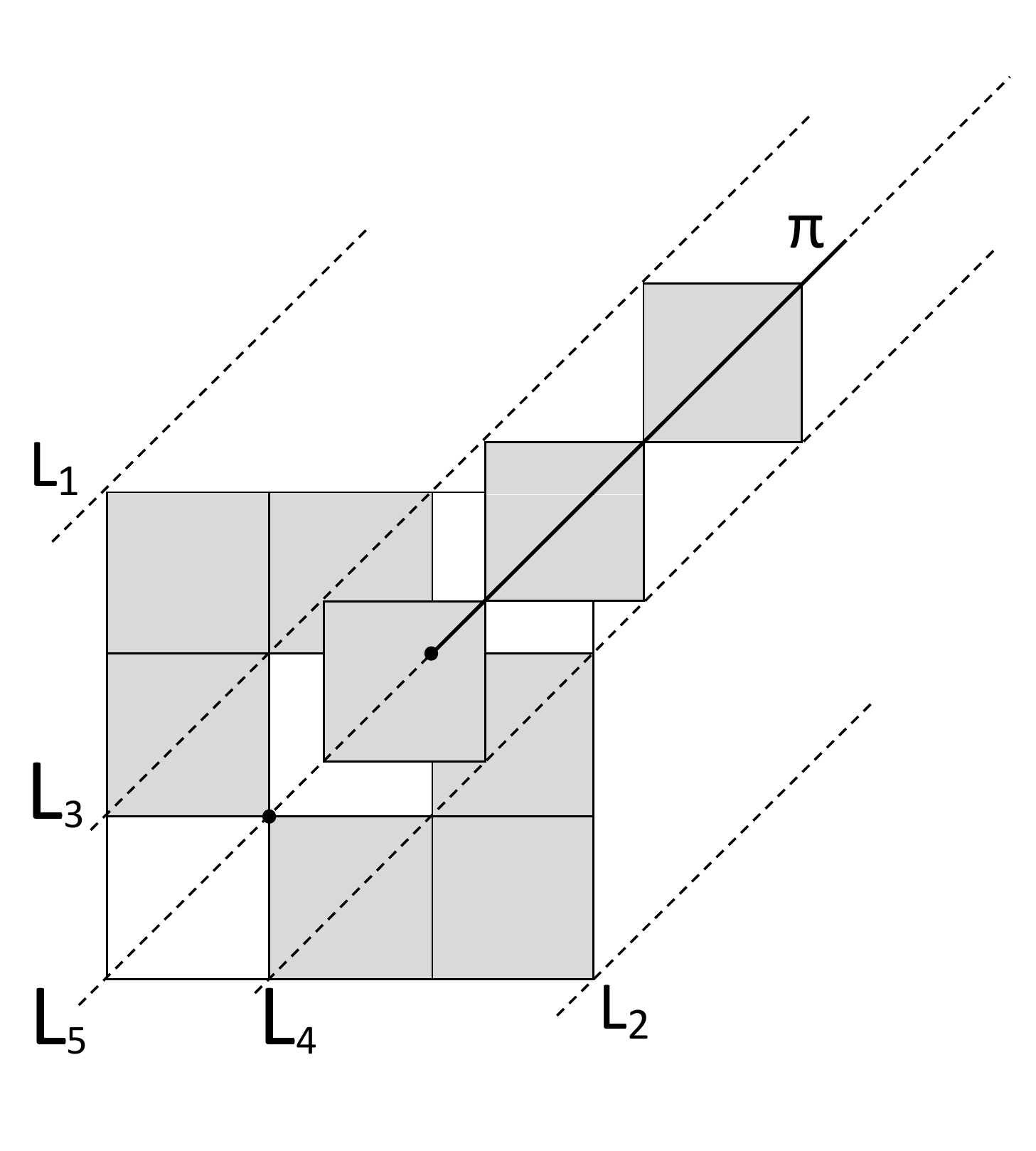}
	}\quad
	\subfigure[]{
		\includegraphics[width=4.2cm, height=4.8cm]{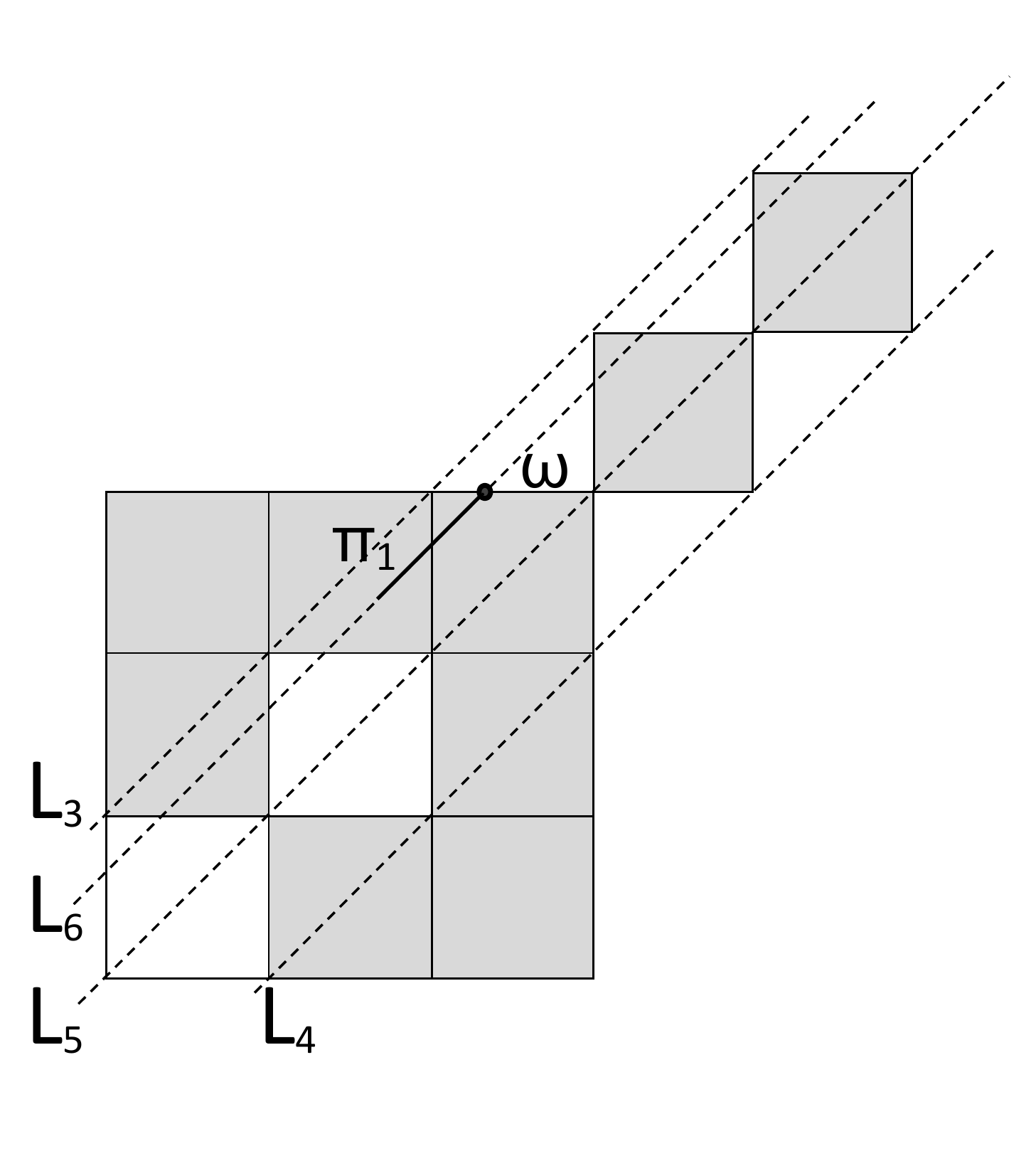}
	}
	\quad
	\subfigure[]{
		\includegraphics[width=4.2cm, height=4.8cm]{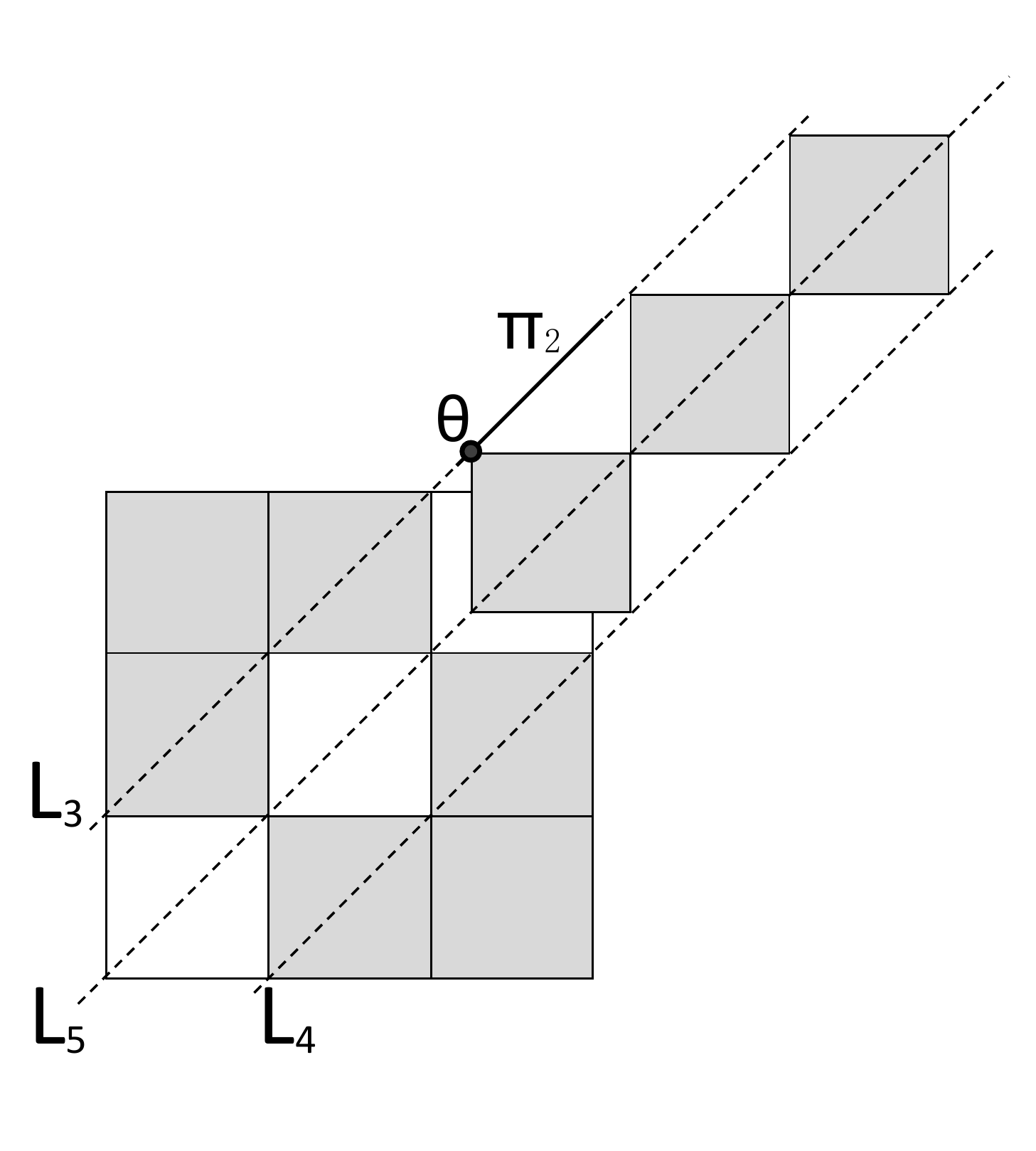}
	}
	\caption{Shifts of diagonal digits along the diagonal line.}\label{fig.2}
\end{figure}

(iv) If $\varepsilon>\frac{(p-1)^2}{p-2}$. Since the right endpoint of $\pi$ has the maximum $x$-coordinate (or $y$-coordinate) in $T_{\varepsilon}$,  $f_{i,j}(\pi)$ has the maximum $x$-coordinate in $f_{i,j}(T_{\varepsilon})$. As the $x$-coordinate of the right endpoint of $f_{i,j}(\pi)$ is $\frac{\varepsilon}{p(p-1)}+\frac{i+1}{p},\ \text{where}\ i<j\ \text{and}\ 0\leq i<p-1$. Then  the right endpoint of $\pi_2$ has the maximum $x$-coordinate, say $\frac{\varepsilon}{p(p-1)}+\frac{p-1}{p}$,  in $V_1$. Note that the point $\theta$ has the minimum $x$-coordinate in $V_2$, and both $\theta$ and $\pi_2$ lie in the same line $L_3$. If  $\varepsilon>\frac{(p-1)^2}{p-2}$, then $\frac{\varepsilon}{p}>\frac{\varepsilon}{p(p-1)}+\frac{p-1}{p}$, implying $\theta\not\in\pi_2$. Hence $V_1\cap V_2=\emptyset$, that is, $T_{\varepsilon}$ is disconnected.

By combining (i), (ii), (iii) and (iv), we prove that: when $\varepsilon\geq 0$, $T_{\varepsilon}$ is connected if and only if $\varepsilon\leq\frac{(p-1)^2}{p-2}$ (see Figure \ref{fig.3}). Similarly for the case that $\varepsilon\leq0$, we can show that $T_{\varepsilon}$ is connected if and only if $\varepsilon\geq -\frac{(p-1)^2}{p-2}$. Therefore, we complete the proof of  Theorem \ref{mainthm3}.
\hfill $\Box$

\begin{figure}[h]
	\centering
	\subfigure[$\varepsilon=3$]{
		\includegraphics[width=4cm, height=4cm]{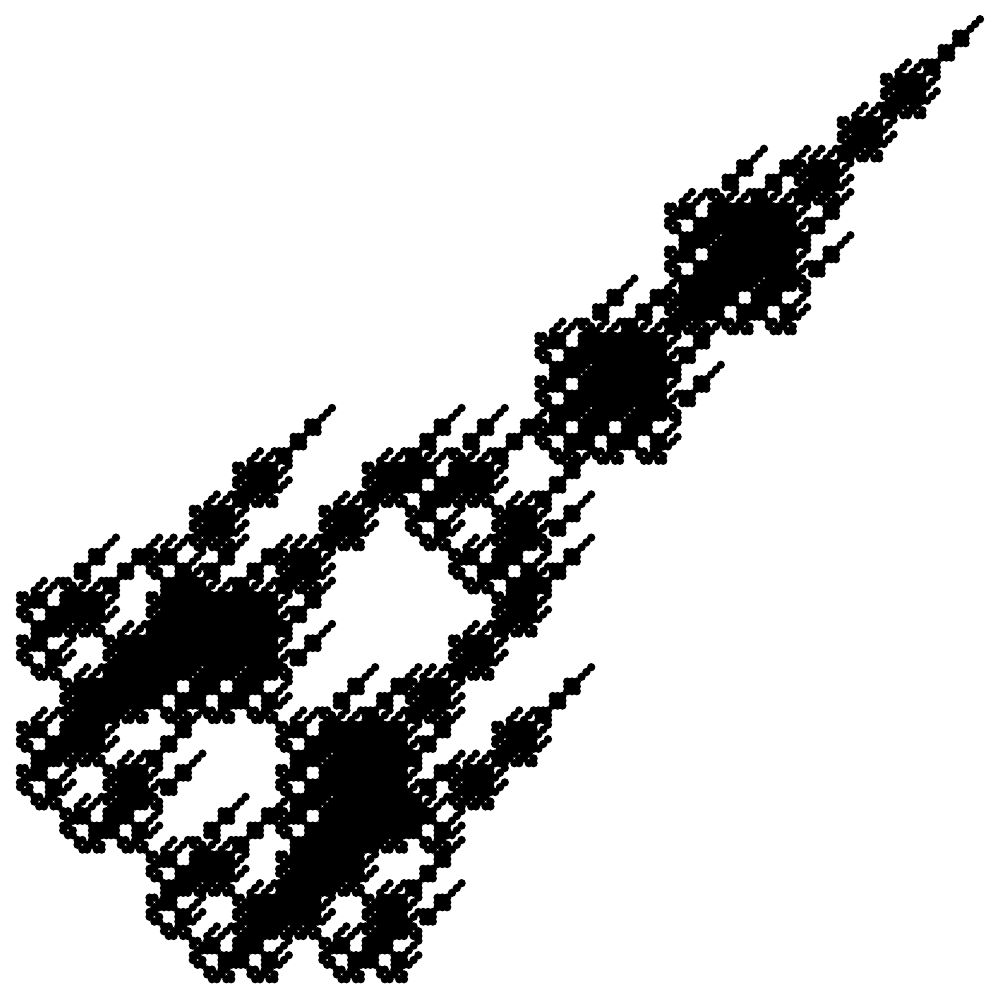}
	}\quad
	\subfigure[$\varepsilon=4$]{
		\includegraphics[width=4cm, height=4cm]{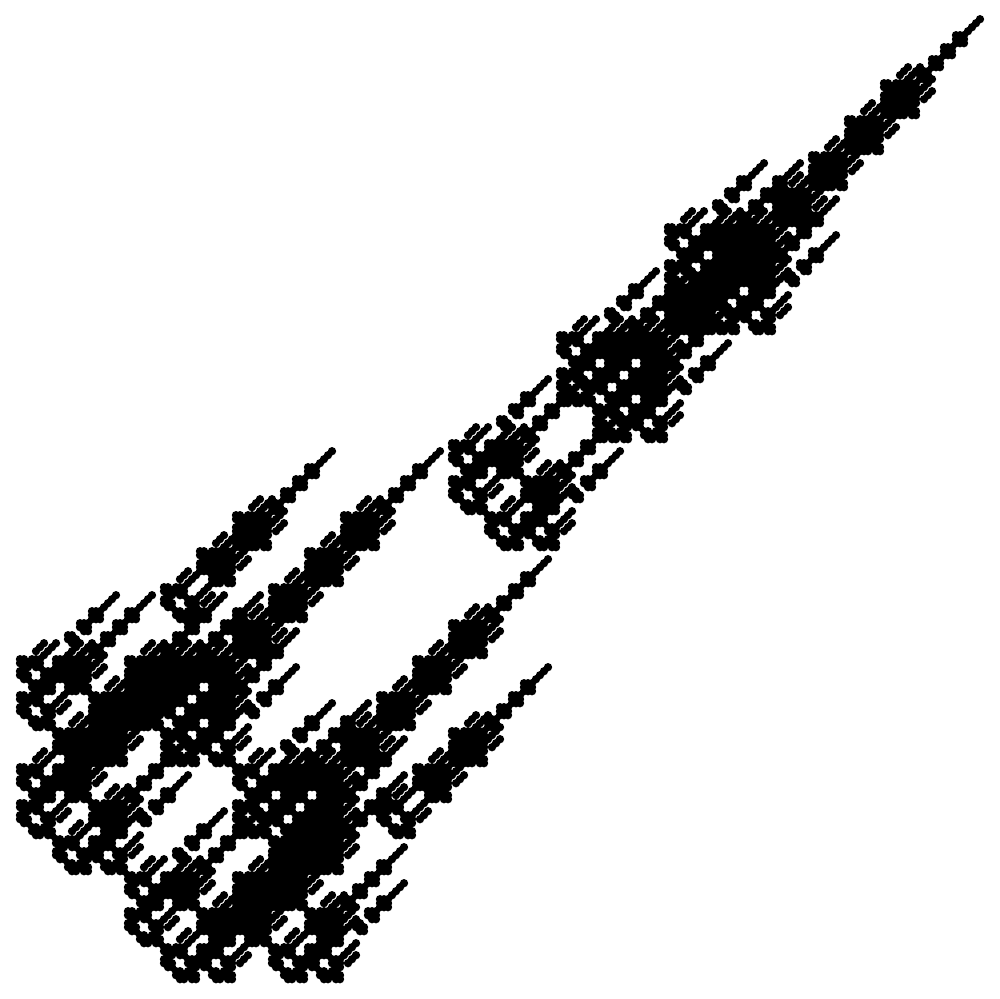}
	}\quad
	\subfigure[$\varepsilon=5$]{
		\includegraphics[width=4cm, height=4cm]{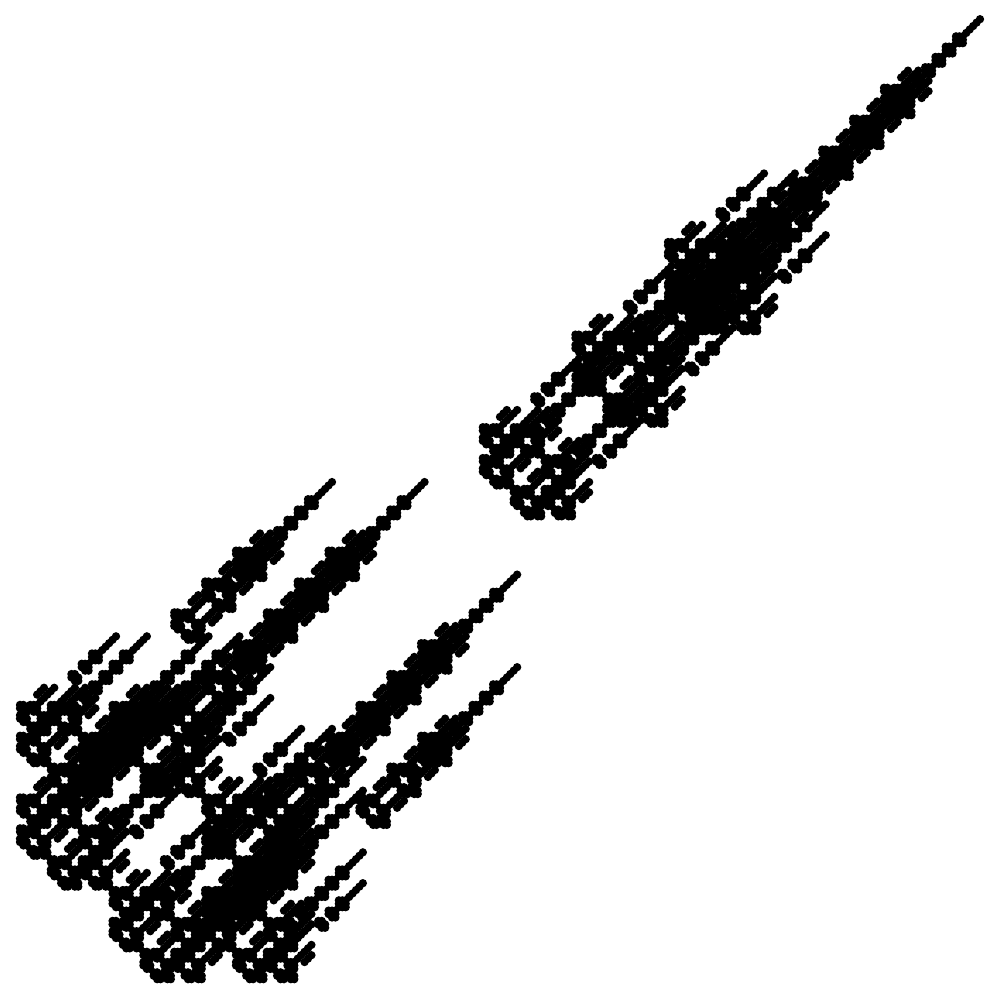}
	}
      \caption{An illustration of Theorem \ref{mainthm3} by taking $p=3$.}\label{fig.3}
\end{figure}

\bigskip

We remark that under the condition of Theorem \ref{mainthm3}, the open set condition does not always hold. For example  if we take $\varepsilon=\ell+k/p$ where $\ell=0,1,\dots,p-2$ and $k=1,2,\dots,p-1$,  we have
$$f_{0,0}\circ f_{0,p-1}\left(\left[\begin{array}{c}x\\y\end{array}\right]\right)=\left[\begin{array}{c}\frac{x}{p^2}+\frac{k}{p^2}+\frac{\ell}{p}\\
\frac{y}{p^2}+\frac{k-1}{p^2}+\frac{\ell+1}{p}\end{array}\right]=f_{\ell,\ell+1}\circ f_{k,k-1}\left(\left[\begin{array}{c}x\\y\end{array}\right]\right).$$
Then $f_{0,0}\circ f_{0,p-1}=f_{\ell,\ell+1}\circ f_{k,k-1}$. Similarly, we have $f_{0,0}\circ f_{p-1,0}=f_{\ell+1,\ell}\circ f_{k-1,k}$ (see Figure \ref{fig.2}(a)). Thus the IFS does not satisfy the open set condition \cite{BG}. Hence $T_{\varepsilon}$ is not a tile in this situation \cite{LaWa}.

\bigskip
\noindent {\bf Acknowledgements:} The first author gratefully acknowledges the support of K. C. Wong Education Foundation and DAAD. The authors also would like to thank  Professor Jun Luo of Sun Yat-Sen Uinversity for many inspiring discussions.


\bigskip

\begin{thebibliography}{99}
\bibliographystyle{ieee}
\addcontentsline{toc}{chapter}{Bibliography}

\bibitem{AkGj}  S. Akiyama and  N. Gjini, {\it Connectedness of number-theoretic tilings}, Discrete Math. Theoret. Computer Science 7 (2005), no. 1, 269-312.

\bibitem{AL}  S. Akiyama and  B. Loridant,  {\it Boundary parametrization of self-affine tiles}, J. Math. Soc. Japan {\bf 63} (2011), no.2, 525-579.

\bibitem{AT} S. Akiyama and J.M. Thuswaldner, {\it A survey on topological properties of tiles related to number systems}, Geom. Dedicata, 109 (2004), 89-105.

\bibitem{BG}  C. Bandt and S. Graf, {\it Self-similar sets $7$: a characterization of self-similar fractals with positive Hausdorff measure}, Proc. Am. Math. Soc. 114 (1992),  995-1001.

\bibitem{BaWa}  C. Bandt and  Y. Wang, {\it Disk-like self-affine tiles in ${\mathbb{R}}^2$},  Discrete Comput. Geom. 26 (2001), no.4, 591-601.

\bibitem{DeLa}  Q.R. Deng and K.S. Lau, {\it Connectedness of a class of planar self-affine tiles}, J. Math. Anal. Appl. 380 (2011), 493-500.

\bibitem{GrHa}  K. Gr\"ochenig and A. Haas, {\it Self-similar lattice tilings}, J. Fourier Anal. Appl. 1 (1994), 131-170.

\bibitem{HaSaVe}  D. Hacon, N.C. Saldanha and J.J.P. Veerman, {\it Remarks on self-affine tilings}, Experiment. Math. 3 (1994), 317-327.

\bibitem{Ha}  M. Hata, {\it On the structure of self-similar sets}, Japan J. Appl. Math. 2 (1985), no.2, 381-414.

\bibitem{Hu} J.E. Hutchinson, {\it Fractals and self-similarity}, Indiana Univ. Math. J. 30 (1981), 713-747.

\bibitem{Ki}  I. Kirat, {\it Disk-like tiles and self-affine curves with non-collinear digits}, Math. Comp. 79 (2010), 1019-1045.

\bibitem{KiLa}  I. Kirat and  K.S. Lau, {\it On the connectedness of self-affine tiles}, J. London Math. Soc. {62} (2000), 291-304.

\bibitem{LaWa}  J.C. Lagarias and Y. Wang, {\it Self-affine tiles in ${\mathbb R}^n$}, Adv. Math. 121 (1996), 21-49.

\bibitem{LaWa1}  J.C. Lagarias and Y. Wang,  {\it Integral self-affine tiles in ${\mathbb R}^n$ I. Standard and nonstandard digit sets},  J. Lond. Math. Soc.  54 (1996) 161-179.

\bibitem{LaWa2} J.C. Lagarias and Y. Wang, {\it Integral Self-affine tiles in ${\mathbb R}^n$ II. Lattice tilings}, J. Fourier Anal. Appl. 3 (1997), 84-102.

\bibitem{LeLa}  K.S. Leung and  K.S. Lau, {\it Disk-likeness of planar self-affine tiles}, Trans. Amer. Math. Soc. {359} (2007), 3337-3355.

\bibitem{LeLu}  K.S. Leung and J.J. Luo, {\it Connectedness of planar self-affine sets associated with non-consecutive collinear digit sets}, J. Math. Anal. Appl. 395 (2012),  208-217.

\bibitem{LeLu2}  K.S. Leung and J.J. Luo, {\it Connectedness of planar self-affine sets associated with non-collinear digit sets}, Geom. Dedicata 175 (2015), 145-157.

\bibitem{LeLu3}  K.S. Leung and J.J. Luo, {\it Boundaries of disk-like self-affine tiles}, Discrete Comput. Geom. 50 (2013), 194-218.

\bibitem{LLY} H. Li, J. Luo and J.D. Yin, {\it The properties of a family of tiles with a parameter}, J. Math. Anal. Appl. 335 (2007), 1383-1396.

\bibitem{LLX} J.C. Liu, J.J. Luo and H.W. Xie, {\it On the connectedness of planar self-affine sets}, Chaos, Solitons \& Fractals 69 (2014), 107-116.

\bibitem{LAT} J. Luo,  S. Akiyama and J.M. Thuswaldner,  {\it On the boundary connectedness of connected tiles}, Math. Proc. Cambridge Philos. Soc. 137 (2004), no.2, 397-410. 

\bibitem{LRT}  J. Luo, H. Rao and B. Tan, {\it Topological structure of self-similar sets}, Fractals 10 (2002), no. 2, 223-227.

\bibitem{NT} S.-M. Ngai and T.-M. Tang, {\it A technique in the topology of connected self-similar tiles}, Fractals 12 (2004), no. 4, 389-403.

\bibitem{NT2} S.-M. Ngai and T.-M. Tang, {\it Topology of connected self-similar tiles in the plane with disconnected interiors}, Topology Appl. 150 (2005), no. 1-3, 139-155.


\end{thebibliography}
\end{document}